\documentclass[12pt,reqno]{article} 

\usepackage[usenames]{color}

\usepackage[colorlinks=true,
linkcolor=webgreen,
filecolor=webbrown,
citecolor=webgreen]{hyperref}

\definecolor{webgreen}{rgb}{0,.5,0}
\definecolor{webbrown}{rgb}{.6,0,0}

\usepackage{graphics,amsmath,amssymb}
\usepackage{amsthm}

\newcommand{\seqnum}[1]{\href{http://www.research.att.com/cgi-bin/access.cgi/as/~njas/sequences/eisA.cgi?Anum=#1}{\underline{#1}}}

\begin{document}


\newcommand{\D}{\displaystyle}
\newcommand{\V}{\vskip.20in}

\begin{center}
\vskip1cm
{\LARGE\bf A Simplified Binet Formula for $k$-Generalized Fibonacci Numbers}
\vskip 1cm
\large
Gregory P.~B.~Dresden\\
Department of Mathematics\\
Washington and Lee University \\
Lexington, VA 24450 \\
\href{mailto:dresdeng@wlu.edu}{\tt dresdeng@wlu.edu} \\
\end{center}

\vskip .2 in

\begin{abstract}
In this paper, we present a particularly nice Binet-style formula that
can be used to produce the $k$-generalized Fibonacci numbers (that is,
the Tribonaccis, Tetranaccis, etc.). Furthermore, we show that in fact
one needs only take the integer closest to the first term of this
Binet-style formula in order to generate the desired sequence. 	
\end{abstract}

\newtheorem{theorem}{Theorem}
\newtheorem{corollary}[theorem]{Corollary}
\newtheorem{lemma}[theorem]{Lemma}
\newtheorem{proposition}[theorem]{Proposition}
\newtheorem{conjecture}[theorem]{Conjecture}

\section{Introduction}

Let $k \geq 2$ and define $F_n^{(k)}$, the 
$n^{\mbox{th}}$ $k$-generalized Fibonacci number, to satisfy the
recurrence relation
\[
F_n^{(k)} = F_{n-1}^{(k)} + F_{n-2}^{(k)} + \dots + F_{n-k}^{(k)}\qquad \qquad (k \mbox{ terms})
\]
... and with initial conditions $0,0,\dots,0,1$ ($k$ terms) such that the 
first non-zero term is $F_1^{(k)} = 1$. 

These numbers are also called the Fibonacci $k$-step numbers, 
Fibonacci $k$-sequences,
or $k$-bonacci numbers. 
Note that for $k=2$, we have $F_n^{(2)} = F_n$, 
our familiar Fibonacci numbers. For $k=3$ we have the so-called 
Tribonaccis 
(sequence number \seqnum{A000073} in 
Sloane's {\it Encyclopedia of
Integer Sequences}),
followed by the Tetranaccis (\seqnum{A000078}) for $k=4$, and so on.
According to Kessler and Schiff \cite{KesslerSchiff}, these 
numbers also appear in probability theory and in certain sorting algorithms. 
We present here a chart of these numbers for the first few values of $k$:
\begin{center}\begin{math}
\begin{array}{l|l|r|l}
k\ \ &\mbox{name} & \mbox{i.c.} & \mbox{\ \ first few non-zero terms} \\
\hline 
2 & \mbox{Fibonacci} & 0,1 & \ \ 1,1,2,3,5,8,13,21,34,\dots \\[1.3ex]
3 & \mbox{Tribonacci} & 0,0,1 &\ \  1,1,2,4,7,13,24,44,81,\dots \\[1.3ex]
4 & \mbox{Tetranacci} & 0,0,0,1 &\ \  1,1,2,4,8,15,29,56,108,\dots \\[1.3ex]
5 & \mbox{Pentanacci} & 0,0,0,0,1 &\ \  1,1,2,4,8,16,31,61,120,\dots
\end{array}
\end{math}\end{center}

We remind the reader of the famous Binet formula (also known
as the de~Moivre formula) that
can be used to calculate $F_n$, the Fibonacci numbers:
\begin{eqnarray*}
F_n &=& \frac{1}{\sqrt{5}} \left[ \left( \frac{1+\sqrt{5}}{2}\right)^n 
     - \left( \frac{1-\sqrt{5}}{2}\right)^n \right] \\[1.5ex]
     &=& \frac{\alpha^n - \beta^n}{\alpha - \beta}
\end{eqnarray*}
...for $\alpha > \beta$ the two roots of $x^2 - x - 1=0$. For our
purposes, it is convenient (and not particularly difficult) to rewrite this
formula as follows:
\begin{equation}
 F_n = \frac{\alpha-1}{2 + 3(\alpha -2)} \alpha^{n-1} + 
     \frac{\beta-1}{2 + 3(\beta -2)} \beta^{n-1}
\label{Binet2}
\end{equation}
We leave the details to the reader.

Our first (and very minor) 
result is the following representation of $F_n^{(k)}$:

\begin{theorem}
\label{T1}
 For $F_n^{(k)}$ the 
$n^{\mbox{th}}$ $k$-generalized Fibonacci number, then
\begin{equation}\label{fT1}
F_n^{(k)} = \sum_{i=1}^k \frac{\alpha_i - 1}{2 + (k+1)(\alpha_i - 2)} \alpha_i^{n-1}
\end{equation}
for $\alpha_1, \dots, \alpha_k$ the roots of $x^k - x^{k-1} - \cdots - 1=0$.
\end{theorem}

This is a new presentation, but hardly a new result. There are many other
ways of representing these $k$-generalized Fibonacci numbers, as seen in 
the articles 
\cite{Ferguson}, \cite{Flores}, \cite{Gabai}, \cite{Kalman}, \cite{LeeKimShin},
\cite{Levesque}, \cite{Miles}.
Our equation 
(\ref{fT1}) of Theorem \ref{T1} is perhaps slightly easier to understand,
and it also allows us to do some analysis (as seen below). 
We point out that for $k=2$, equation (\ref{fT1}) reduces to 
the variant of the Binet formula (for the standard Fibonacci numbers) from
equation (\ref{Binet2}). 

As shown in three distinct proofs 
(\cite{Miles}, \cite{Miller}, and \cite{Wolfram}), the equation 
$x^k - x^{k-1} - \cdots - 1=0$ from Theorem \ref{T1} has just one 
root $\alpha$ such that
$|\alpha|>1$, and
the other roots are strictly inside the unit circle. We can
conclude that the contribution of the other 
roots in formula \ref{fT1} will quickly become trivial, and thus:
\begin{equation}\label{obs}
F_n^{(k)} \approx \frac{\alpha - 1}{2 + (k+1)(\alpha - 2)} \alpha^{n-1} 
\qquad\mbox{... for $n$ sufficiently large.}
\end{equation}
It's well known that 
for the Fibonacci sequence $F_n^{(2)} = F_n$, the
``sufficiently large'' $n$ in equation (\ref{obs}) is $n=0$, as shown
here:

\begin{center}\begin{math}
\begin{array}{c|ccccccc}
n & 0 & 1 & 2 & 3 & 4 & 5 & 6  \\
\hline
F_n & 0 & 1 & 1 & 2 & 3 & 5 & 8 \\[1.3ex]
\frac{1}{\sqrt5}\left(\frac{1 + \sqrt5}{2}\right)^n 
   & 0.447 & 0.724 & 1.171 & 1.894 & 3.065 &  4.960 & 8.025 \\[1.3ex]
|\mbox{error}| & .447 & .277 & .171 & .106 & .065 & .040 & .025
\end{array}
\end{math}\end{center}
It is perhaps surprising to discover that a similar statement holds 
for all the $k$-generalized Fibonacci numbers. Our main result is the 
following:

\begin{theorem}\label{T2}
For $F_n^{(k)}$ the 
$n^{\mbox{th}}$ $k$-generalized Fibonacci number, then
\[
F_n^{(k)} = \mbox{\rm Round}\left[ \frac{\alpha - 1}{2 + (k+1)(\alpha - 2)} \alpha^{n-1} \right]
\]
for all $n \geq 2-k$ and for $\alpha$ the unique 
positive root of $x^k - x^{k-1} - \cdots - 1=0$.
\end{theorem}
We point out that this theorem is not as trivial as one might think. 
Note the error 
for $k=6$, as seen in the following chart; it is not monotone decreasing.

\begin{center}\begin{math}
\begin{array}{c|ccccccccccc}
n & 0 & 1 & 2 & 3 & 4 & 5 & 6 & 7  \\
\hline
F_n^{(6)} & 0 & 1 & 1 & 2 & 4 & 8 & 16 & 32  \\[1.3ex]
\frac{\alpha -1}{2 + 7(\alpha-2)}\alpha^5
  &   0.263 & 0.522 & 1.035 & 2.053 & 4.072 & 8.078 & 16.023 & 31.782  \\[1.3ex]
|\mbox{error}| & 
.263 & .478 & .035 & .053 & .072 & .078 & .023 & .218 
\end{array}
\end{math}\end{center}
We also point out that not every recurrence sequence admits such a nice
formula as seen in Theorem \ref{T2}. Consider, for example, the
scaled Fibonacci sequence $10, 10, 20, 30, 50, 80, \dots$, 
which has Binet formula:
\[
\frac{10}{\sqrt{5}}\left( \frac{1+\sqrt{5}}{2}\right)^n 
     - 
\frac{10}{\sqrt{5}}\left( \frac{1-\sqrt{5}}{2}\right)^n.
\]
This can be written as
$
\mbox{\rm Round}\left[\frac{10}{\sqrt{5}}\left( \frac{1+\sqrt{5}}{2}\right)^n \right]
$,
but only for $n \geq 5$. 
As another example, the sequence $1,2,8,24,80, \dots$ (defined by 
$G_n = 2G_{n-1} + 4G_{n-2}$)
can be written as
\[
G_n =  \frac{(1+\sqrt{5})^n}{2 \sqrt{5}}
     - 
\frac{(1-\sqrt{5})^n}{2 \sqrt{5}},
\]
but because both $1+\sqrt{5}$ and $1-\sqrt{5}$ have absolute value 
greater than $1$, then it would be impossible to express $G_n$ in terms
of just one of these two numbers. 


\section{Previous Results}

We point out that for $k=3$ (the Tribonacci numbers), our Theorem
\ref{T2} was found earlier by Spickerman \cite{Spickerman}. 
His formula 
(modified slightly to match our
notation) reads as follows, where $\alpha$ is the real root, and
$\sigma$ and $\overline{\sigma}$ are the two complex roots, of 
$x^3-x^2-x-1 = 0$:
\begin{equation}\label{sr}
F_n^{(3)} = \mbox{\rm Round}\left[ \frac{\alpha^2}{(\alpha - \sigma)(\alpha - \overline{\sigma})} \alpha^{n-1} \right]
\end{equation}
It is not hard to show that 
for $k=3$, 
our coefficient
$\frac{\alpha - 1}{2 + (k+1)(\alpha - 2)}$
 from Theorem \ref{T2} is equal to  
Spickerman's  coefficient 
$ \frac{\alpha^2}{(\alpha - \sigma)(\alpha - \overline{\sigma})}$.
We leave the details to the reader. 

In a subsequent article \cite{SJ}, Spickerman and Joyner developed
a more complex version of our Theorem \ref{T1} to represent the 
generalized Fibonacci numbers. 
Using our notation, and with $\{\alpha_i\}$ the set of roots of
$x^k - x^{k-1} - \cdots - 1 = 0$, their fomula reads
\begin{equation}\label{sjf}
F_n^{(k)} = \sum_{i=1}^k 
       \frac{\alpha_i^{k+1} - \alpha_i^{k}}
            {2\alpha_i^k - (k+1)} \alpha_i^{n-1}
\end{equation}
It is surprising that even after calculating out the appropriate constants
in their equation (\ref{sjf}) for $2 \leq k \leq 10$, 
neither Spickerman nor Joyner noted that they could have
simply taken the first term in equation (\ref{sjf}) for all $n \geq 0$,
as Spickerman did in equation (\ref{sr}) for $k=3$.

The Spickerman-Joyner formula  (\ref{sjf})
was extended by Wolfram \cite{Wolfram} to the case with
arbitary starting conditions (rather than the initial
sequence $0, 0, \dots, 0, 1$). 
In the next section we will show that 
our formula (\ref{fT1}) in Theorem \ref{T1} is
equivalent to the Spickerman-Joyner formula given above (and thus is 
a special case of Wolfram's formula). 

Finally, we note that the polynomials 
 $x^k - x^{k-1} - \cdots - 1$ in Theorem \ref{T1}
 have been studied rather extensively.
They are irreducible polynomials with just 
%
%
one zero outside the unit circle. That single zero is 
located between $2(1- 2^{-k})$ and $2$ (as seen in Wolfram's article
\cite{Wolfram}; Miles \cite{Miles} gave earlier and
less precise results). 
It is also known \cite[Lemma 3.11]{Wolfram} that the polynomials have
Galois group $S_k$ for $k \leq 11$; in particular, their zeros
can not be expressed in radicals for $5 \leq k \leq 11$. Wolfram
conjectured that the Galois group is always $S_k$. Cipu and Luca \cite{CipuLuca}
were able to show that the Galois group
is not contained in the alternating group $A_k$,
and for $k \geq 3$ 
it is not $2$-nilpotent. 
They point out that this means the zeros of the polynomials
 $x^k - x^{k-1} - \cdots - 1$ for $k \geq 3$ can not 
 be constructed by ruler and compass, but the question of whether they
 are expressible using radicals remains open.

\section{Preliminary Lemmas}

First, a few statements about the the number $\alpha$.

\begin{lemma}\label{La}
Let $\alpha>1$ be the 
real positive root of $x^k - x^{k-1} - \cdots - x - 1=0$. Then, 
\begin{equation}
2 - \frac{1}{k} <  \alpha < 2 
\end{equation}
In addition,
\begin{equation}
\ \ \ \ \ \ \ \ \ \ \ \ \   2 - \frac{1}{3k} <  \alpha < 2 \qquad \mbox{for } k \geq 4
\end{equation}
\end{lemma}

\begin{proof} 
We begin by computing the following chart for $k \leq 5$:
\begin{center}\begin{math}
\begin{array}{c|llc}
k & 2 - \frac{1}{k} & 2 - \frac{1}{3k} &  \alpha \\[1.2ex]
\hline
2 & 1.5   & 1.833\dots & 1.618\dots \\
3 & 1.666\dots \ \ \ & 1.889\dots \ \ \  & 1.839\dots \\
4 & 1.75  & 1.916\dots & 1.928\dots \\
5 & 1.8   & 1.933\dots & 1.966\dots 
\end{array}
\end{math}\end{center}

It's clear that $2 - \frac{1}{k} <  \alpha < 2$ for $2 \leq k \leq 5$ and that
$2 - \frac{1}{3k} <  \alpha < 2$ for $4 \leq k \leq 5$.
We now focus on $k \geq 6$. At this point, we 
could finish the proof by appealing to
$2(1 - 2^{-k}) < \alpha < 2$ as seen in the article \cite[Lemma 3.6]{Wolfram}, 
but we present here a simpler proof. 

Let $f(x) = (x-1)(x^k - x^{k-1} - \cdots - x - 1) = x^{k+1} -2x^k +1$. We know
from our earlier discussion that $f(x)$ has one real zero $\alpha > 1$. 
Writing $f(x)$ as $x^k(x-2) +1$, we have
\begin{equation}\label{23k}
f\left(2 - \frac{1}{3k}\right) = \left(2 - \frac{1}{3k}\right)^k 
         \left( \frac{-1}{3k}\right) + 1
\end{equation}
For $k\geq 6$, it's easy to show
\[
3k < \left(\frac{5}{3}\right)^k = \left(2 - \frac{1}{3}\right)^k < 
 \left(2 - \frac{1}{3k}\right)^k
\]
Substituting this inequality into the right-hand side of (\ref{23k}), we can
re-write 
 (\ref{23k}) as:
\[
f\left(2 - \frac{1}{3k}\right) < (3k)\cdot\left(\frac{-1}{3k}\right)  + 1 = 0.
\]
Finally, we note that
\[  f(2) = 2^{k+1} - 2\cdot 2^k + 1 = 1 >0,  \] 
so we can conclude that
our root $\alpha$ is within the desired bounds of $2 - 1/3k$ and $2$ for 
$k\geq 6$. 
\end{proof}

We now have a lemma about the coefficients of $\alpha^{n-1}$ in Theorems \ref{T1}
and \ref{T2}.

\begin{lemma}\label{Lm}
Let $k \geq 2$ be an integer, and let 
$m^{(k)}(x) = \D \frac{x-1}{2 + (k+1)(x-2)}$. Then, 
\begin{enumerate} 
\item  $m^{(k)}(2 - 1/k) = 1$. \label{Lmp1}
\item  $m^{(k)}(2) = \frac{1}{2}$.  \label{Lmp2}
\item  $m^{(k)}(x)$ is continuous and decreasing on the interval
$[2- 1/k, \infty)$. \label{Lmp3}
\item  $m^{(k)}(x) > \frac{1}{x}$ on the interval
$(2- 1/k, 2)$. \label{Lmp4}
\end{enumerate} 
\end{lemma}

\begin{proof} 
Parts \ref{Lmp1} and \ref{Lmp2} are immediate. As for \ref{Lmp3}, note that 
we can rewrite $m^{(k)}(x)$ as:
\[
m^{(k)}(x) = \frac{1}{k+1} 
  \left[ 1 + \frac{1 - \frac{2}{k+1}}{x - (2 - \frac{2}{k+1})} \right]
\]
which is simply a scaled translation of the map $y = 1/x$. In particular, since
this $m^{(k)}(x)$ has a vertical asymptote at $x = 2 - \frac{2}{k+1}$, then
by parts \ref{Lmp1} and \ref{Lmp2} we can conclude that $m^{(k)}(x)$ is 
indeed continuous and decreasing on the desired interval. 

To show part \ref{Lmp4}, we first note that in solving
$\frac{1}{x} = m^{(k)}(x)$, we obtain a quadratic equation with the 
two intersection points $x=2$ and $x=k$. 
It's easy to show that 
$\frac{1}{x} < m^{(k)}(x)$  at 
$x = 2 - 1/k$, and since both functions $\frac{1}{x}$ and $m^{(k)}(x)$
are continuous on the inverval 
$[2- 1/k, \infty)$  and intersect only at $x=2$ and $x=k \geq 2$, we 
can conclude that $\frac{1}{x} < m^{(k)}(x)$  on the desired interval.
\end{proof}

\begin{lemma}\label{Le}
For a fixed value of $k \geq 2$ and for $n \geq 2-k$, 
define $E_n$ to be the error in our
Binet approximation of Theorem \ref{T2}, as follows:
\begin{eqnarray*}
E_n &=& F_n^{(k)} - \frac{\alpha - 1}{2 + (k+1)(\alpha - 2)} \cdot \alpha^{n-1}\\
    &=& F_n^{(k)} - m^{(k)}(\alpha) \cdot \alpha^{n-1},
\end{eqnarray*}
... for $\alpha$ the positive real root of 
$x^k - x^{k-1} - \cdots - x - 1 = 0$ and
$m^{(k)}$ as defined in Lemma \ref{Lm}. Then, $E_n$ satisfies the same recurrence
relation as $F_n^{(k)}$:
\[
E_n = E_{n-1} + E_{n-2} + \dots + E_{n-k} \qquad \qquad (\mbox{for } n \geq 2)
\]
\end{lemma}

\begin{proof}
By definition, we know that $F_n^{(k)}$ satisfies the recurrence relation:
\begin{equation} \label{Le1}
F_n^{(k)} = F_{n-1}^{(k)} + \dots + F_{n-k}^{(k)}
\end{equation}
As for the term $m^{(k)}(\alpha) \cdot \alpha^{n-1}$, note that $\alpha$ is a root
of $x^k - x^{k-1} - \cdots - 1 = 0$, which means that
$\alpha^k =  \alpha^{k-1} + \cdots + 1$, which implies
\begin{equation} \label{Le2}
m^{(k)}(\alpha) \cdot \alpha^{n-1} =  
m^{(k)}(\alpha)\alpha^{n-2} + \cdots + m^{(k)}(\alpha)\alpha^{n-(k+1)}
\end{equation}
We combine Equations (\ref{Le1}) and (\ref{Le2}) to obtain the desired result. 
\end{proof}


\section{Proof of Theorem \ref{T1}}

As mentioned above, Spickerman and Joyner \cite{SJ} proved the following
formula for the $k$-generalized Fibonacci numbers:
\begin{equation}\label{sjf2}
F_n^{(k)} = \sum_{i=1}^k 
       \frac{\alpha_i^{k+1} - \alpha_i^{k}}
            {2\alpha_i^k - (k+1)} \alpha_i^{n-1}
\end{equation}
Recall that the set $\{\alpha_i\}$ is the set of roots of 
$x^k - x^{k-1} - \cdots - 1=0$. We now show that this formula is equivalent to 
our equation (\ref{fT1}) in Theorem \ref{T1}:
\begin{equation}\label{fT1again}
F_n^{(k)} = \sum_{i=1}^k \frac{\alpha_i - 1}{2 + (k+1)(\alpha_i - 2)} \alpha_i^{n-1}
\end{equation}
Since $\alpha_i^k - \alpha_i^{k-1} - \cdots - 1 = 0$, 
we can multiply by $\alpha_i -1$ to get
$\alpha_i^{k+1} - 2\alpha_i^{k} = -1$, which implies
$(\alpha_i-2) = -1\cdot\alpha_i^{-k}$.
We use this last equation to transform  (\ref{fT1again}) as follows:
\[
\frac{\alpha_i - 1}{2 + (k+1)(\alpha_i - 2)} =
\frac{\alpha_i - 1}{2 + (k+1)(-\alpha_i^{-k})} = 
    \frac{\alpha_i^{k+1} - \alpha_i^{k}}{2\alpha_i^k - (k+1)}
\]
This establishes the equivalence of the two formulas
(\ref{sjf2}) and (\ref{fT1again}), as desired. 
\hfill \ \ $\Box$


\section{Proof of Theorem \ref{T2}}

Let $E_n$ be as defined in Lemma \ref{Le}. 
We wish to show that $|E_n| < \frac{1}{2}$ for all $n \geq 2-k$. 
We proceed by
first showing that $|E_n| < \frac{1}{2}$ for $n=0$, then for 
$n= -1, -2, -3,  \dots, 2-k$, then for $n=1$, and finally that 
this implies $|E_n| < \frac{1}{2}$ for all $n \geq 2-k$. 

To begin, we note that since our initial conditions give us that
$F_n^{(k)} = 0$ for 
$n = 0, -1, -2, \dots, 2-k$, then we need only show
$|m^{(k)}(\alpha) \cdot \alpha^{n-1}| < 1/2$ for 
those values of $n$. Starting with $n=0$,
it's easy to check by hand that 
$\D m^{(k)}(\alpha) \cdot \alpha^{-1} < 1/2$ for $k=2$ and $3$, and 
as for $k \geq 4$,
we have the following inequality from Lemma \ref{La}:
\[
2 - \frac{1}{3k} < \alpha,
\]
which implies
\[
{\alpha^{-1}} < \frac{3k}{6k-1}.
\]
Also, by Lemma \ref{Lm},
\[ 
m^{(k)}(\alpha) < m^{(k)}(2 - 1/{3k}) =  \frac{3k-1}{5k-1},
\] so thus:
\[
m^{(k)}(\alpha) \cdot \alpha^{-1}  < \frac{3k-1}{5k-1} \cdot  \frac{3k}{6k-1} 
 < \frac{(3k) \cdot 1}{(5k-1)\cdot 2} < \frac{1}{2},
\]
as desired. 
Thus, $0< |m^{(k)}(\alpha) \cdot \alpha^{-1}| < 1/2$ for all $k$, as desired.

Since $\alpha^{-1} < 1$, we can conclude that for $n= -1, -2, \dots, 2-k$, 
then
$|E_n| = m^{(k)}(\alpha) \cdot \alpha^{n-1} < 1/2$.

Turning our attention now to $E_1$, we note that $F_1^{(k)} = 1$
(again by definition of our initial conditions) and that
\[
 \frac{1}{2} = m(2) <  m(\alpha) <  m(2 - 1/k) = 1
\]
which immediately gives us $|E_1| < 1/2$. 

As for $E_n$ with $n \geq 2$, we know from Lemma \ref{Le} that
\[
E_n = E_{n-1} + E_{n-2} + \dots + E_{n-k} \qquad \qquad (\mbox{for } n \geq 2)
\]
Suppose for some $n \geq 2$ that $|E_n| \geq 1/2$. Let $n_0$ be the smallest
positive such $n$. Now, subtracting the following two equations:
\begin{eqnarray*}
E_{n_0+1} &=& E_{n_0} + E_{n_0-1} + \dots + E_{n_0-(k-1)} \\
E_{n_0} &=& E_{n_0-1} + E_{n_0-2} + \dots + E_{n_0-k}
\end{eqnarray*}
gives us:
\[
E_{n_0+1} = 2E_{n_0} - E_{n_0-k} 
\]
Since $|E_{n_0}| \geq |E_{n_0 -k}|$ 
(the first, by assumption, being larger than, and the second smaller than, $1/2$),
we can conclude that
$|E_{n_0+1}| > |E_{n_0}|$. 
In fact, we can apply this argument repeatedly to show that
$|E_{n_0+i}| > \cdots > |E_{n_0 + 1}| >  |E_{n_0}|$. 
However, this contradicts the observation from 
equation (\ref{obs}) that 
the error must eventually go to $0$. We conclude that 
$|E_n| < 1/2$ for all $n \geq 2$, and thus for all $n \geq 2-k$.
\hfill \ \ $\Box$

\section{Acknowledgement} 
The author would like to thank 
J.~Siehler for inspiring this paper with his work on Tribonacci numbers.

 
%
%


\bigskip
\hrule
\bigskip

\noindent 2000 {\it Mathematics Subject Classification}:
Primary 11B39, Secondary 11C08, 33F05, 65D20.

\noindent \emph{Keywords: } $k$-generalized Fibonacci numbers, 
Binet, Tribonacci, Tetranacci, Pentanacci.

\bigskip
\hrule
\bigskip

\noindent (Concerned with sequences
\seqnum{A000073}, 
\seqnum{A000078}, 
and \seqnum{A001591}.)

\bigskip
\hrule
\bigskip

\vspace*{+.1in}
\noindent
Received October XX, 2008.

\bigskip
\hrule
\bigskip

\noindent
Return to
\htmladdnormallink{Journal of Integer Sequences home page}{http://www.cs.uwaterloo.ca/journals/JIS/}.
\vskip .1in

\end{document}